\documentclass[11pt,reqno]{amsart} 

\usepackage[utf8]{inputenc} 
\usepackage{bbm}

\usepackage[margin=1in]{geometry} 

\usepackage{graphicx} 
\usepackage{float} 

 \usepackage[parfill]{parskip} 
 
\usepackage{booktabs} 
\usepackage{array} 
\usepackage{paralist} 
\usepackage{verbatim} 
\usepackage{subfig} 
\usepackage{mathrsfs}
\usepackage{amssymb}
\usepackage{amsthm}
\usepackage{amsmath,amsfonts,amssymb,esint}
\usepackage{graphics,color}
\usepackage{enumerate, enumitem}
\usepackage{mathtools,centernot}
\usepackage{cases}
\usepackage{amsrefs}

\usepackage{xfrac}

\newtheorem{theorem}{Theorem}[section]

\newtheorem{corollary}[theorem]{Corollary}

\newtheorem{proposition}[theorem]{Proposition}

\numberwithin{equation}{section}

\newcommand{\norm}[1]{\left\|#1\right\|}

\newcommand{\T}{\ensuremath{\mathbb{T}}}
\newcommand*{\R}{\ensuremath{\mathbb{R}}}

\newcommand*{\Z}{\ensuremath{\mathbb{Z}}}

\usepackage{color, graphicx}
\usepackage{mathrsfs, dsfont}

\usepackage{hyperref}

\def\dist{\mathop{\rm dist}\nolimits}    
\def\div{\mathop{\rm div}\nolimits}    
\def\dim{\mathop{\rm dim}\nolimits}

\DeclareMathOperator*{\esssup}{ess\,sup}

\title{A fractal version of the Onsager's conjecture: the $\beta-$model}

\author{Luigi De Rosa}
\address{Department Mathematik Und Informatik, Universitat Basel, Spiegelgasse 1, CH-4051 Basel, Switzerland}
\email{luigi.derosa@unibas.ch}

\author{Silja Haffter}
\address{EPFL SB, Station 8, 
CH-1015 Lausanne, Switzerland}
\email{silja.haffter@epfl.ch}

\date{\today}
\begin{document}

\begin{abstract}
Intermittency phenomena are known to be among the main reasons why Kolmogorov's theory of fully developed Turbulence is not in accordance with several experimental results. This is why some \emph{fractal} statistical models have been proposed in order to realign the theoretical physical predictions with the empirical experiments. They indicate that energy dissipation, and thus singularities, are not space filling for high Reynolds numbers. 
This note aims to give a precise mathematical statement on the energy conservation of such fractal models of Turbulence. We prove that 
for $\theta-$H\"older continuous weak solutions of the incompressible Euler equations energy conservation holds if the upper Minkowski dimension of the spatial singular set  $S \subseteq \T^3$ (possibly also time-dependent) is small, or more precisely if $\overline{\dim}_{\mathcal{M}}(S)<2+3\theta\,.$ In particular, the spatial singularities of \emph{non-conservative} $\theta-$H\"older continuous weak solutions of Euler are concentrated on a set with dimension lower bound $2+3\theta$. This result can be viewed as the fractal counterpart of the celebrated Onsager conjecture and it matches both with the prediction given by the $\beta-$model introduced by Frisch,  Sulem and Nelkin in \cite{FSN78} and with  other mathematical results in the endpoint cases. 
\end{abstract}

\maketitle

\par
\medskip\noindent
{\textbf{MSC (2020):} 	35Q31 - 35D30  - 76F05 - 28A80.
\par
}

\section{Introduction}

We consider the incompressible Euler equations 
\begin{equation}\label{E}
\left\{\begin{array}{l}
\partial_t v+\div (v\otimes v)  +\nabla p =0\\
\div v = 0,
\end{array}\right.
\end{equation}
in the spatial periodic setting $\T^3=\R^3\setminus \Z^3$,
where $v:\T^3\times [0,T]\rightarrow \R^3$ is a vector field representing the velocity of the fluid and $p:\T^3\times [0,T]\rightarrow \R$ is the hydrodynamic pressure. We study weak solutions of the system \eqref{E}, namely vector fields $v\in L^2\left(\T^3\times [0,T];\R^3\right)$ such that 
$$
\int_0^T\int_{\T^3}\left( v\cdot \partial_t \varphi+ v\otimes v : \nabla \varphi \right)\,dx\, dt=0
$$
for all test functions $\varphi\in C^\infty_c\left(\T^3\times (0,T);\R^3\right)$ with $\div \varphi=0$.  The pressure does not appear in the weak formulation due to the constraint $\div \varphi=0$, but it can be recovered a posteriori from the weak solution $v$ as the unique $0-$average solution of the elliptic equation
\begin{equation}\label{Lapl_p}
-\Delta p=\div \div (v\otimes v)\,,
\end{equation}
which is derived, formally, by computing the divergence of the first equation in \eqref{E}.

The Euler system models the motion of an incompressible non-viscous fluid and it can be seen as the vanishing viscosity limit of the Navier-Stokes equations, where the viscosity modelled with the additional term $\nu \Delta v$ for $\nu>0$ on the right-hand side of the first equation in \eqref{E}. This term is responsible for the kinetic energy dissipation
\begin{equation}\label{Energy_NS}
\frac{d}{dt} e_{v^\nu}(t)=-\nu \int_{\T^3} |\nabla v^\nu(x,t)|^2\,dx
\end{equation}
of solutions $v^\nu$ of the Navier-Stokes equations, where the kinetic energy is defined as
$$
e_v(t):=\frac{1}{2}\int_{\T^3} |v(x,t)|^2\,dx.
$$
The energy identity \eqref{Energy_NS} holds for sufficiently smooth solutions $v^\nu$ of the Navier-Stokes equations and,  in the \emph{lucky} case in which their regularity (say $C^1(\T^3 \times [0,T])$) is preserved in the limit $\nu\rightarrow 0$, this would imply that the vanishing viscosity limit $v$ is a conservative solution of Euler, i.e. 
\begin{equation}\label{en_cons}
\frac{d}{dt}e_v\equiv 0 \,.
\end{equation}
Kolmogorov's celebrated Theory of Turbulence (K-41) \cite{K41} from 1941
instead postulates that the mean energy dissipation rate $\nu \int |\nabla v^\nu |\,dx$ remains, in a suitable statistical sense, uniformly positive as $\nu \rightarrow 0$,  thus predicting that in general the smoothness of solutions $v^\nu$ to Navier-Stokes deteriorates in the vanishing viscosity limit.  This prediction culminated in Onsager's conjecture \cite{Ons} from 1949, which asserted that the sharp  regularity threshold to deduce energy conservation \eqref{en_cons} for solutions $v$ of \eqref{E} is $v \in L^\infty\left((0,T);C^{\sfrac{1}{3}}(\T^3)\right)$.

The full proof of energy conservation for weak solutions $v$ of \eqref{E} was given by Constantin, E and Titi \cite{CET94} for $v \in L^3\left((0,T);B^\theta_{3,\infty}(\T^3)\right)$ for some $\theta>\sfrac13$ (see also \cite{Ey94}).  We refer also to \cite{CCFS08} for a slightly sharper result and to \cite{CD18,Is2013} for a different proof relying on the kinetic energy regularity $e_v\in C^{\sfrac{2\theta}{(1-\theta)}}([0,T])$. While the energy conservation relies on somehow standard analytical tools, the construction of non-conservative solutions to \eqref{E} turned out to be challenging and technical.  The first $L^2$ weak solutions to \eqref{E} violating \eqref{en_cons} were built in \cite{Sch93} and \cite{Sh00} respectively by Scheffer and Shnirelman. The turning point in resolving the flexibility part of Onsager's conjecture was reached with the work \cite{DS2013} of De Lellis and Székelyhidi in which they proved the existence of $C^0$ dissipative Euler flows by introducing so-called \emph{convex integration} techniques in the context of fluid dynamics. After a series of fundamental advancements \cite{BDLIS15}, \cite{DS17} and  \cite{BDS16}, Isett finally proved the existence of non-conservative $\theta-$H\"older continuous weak solutions of Euler in the whole range $\theta<\frac13$ in \cite{Is2018}, thus confirming Onsager's prediction (see also \cite{BDLSV2019} for an improvement to \emph{dissipative} solutions).

Kolmogorov's Theory of Turbulence, and consequently Onsager's prediction, builds on the assumption of statistical 
homogeneity, isotropy and self-similarity, which implies, by scaling arguments and a dimensional analysis, that the structure functions $S_p(\ell):=\langle|v(x+\ell \vec{e} , t)-v(x, t)|^p \rangle$ obey
$$
S_p(\ell)\simeq \ell ^{\zeta_p},
$$
where $\zeta_{p}=\frac{p}{3}$, $p\geq 1$ and $\langle \cdot \rangle$ denotes an average over $x \in \T^3$, $\vec{e} \in \mathbb{S}^2$ and a probabilistic ensemble average that, in case of ergodicity of the random process governing the fluid motion, coincides with a long time average.  Thus Kolmogorov's prediction $\zeta_p=\frac{p}{3}$ gives the rather natural candidate $\theta=\frac{1}{3}$ as a H\"older regularity exponent for such turbulent solutions as noticed by Onsager.
It is well-known though that away from the case $p=3$, in which $\zeta_{p}=1$ is an exact result known as Kolomogorov's $\sfrac{4}{5}$-law, $\zeta_p=\frac{p}{3}$ is in disagreement with several experiments  \cite{BJPV98},  \cite{F95}. These phenomena go under the name of \emph{intermittency} and are a consequence of the breakdown of self-similarity of Turbulence (see for instance \cite{ET99}), which suggests that the exponents $\zeta_p$ can not be determined from a scaling analysis alone.  Furthermore, it has been observed numerically \cite{S81} that  at high Reynolds numbers, singularities (or large gradients) of solutions, and thus also energy dissipation, are not space filling and thus one should assume that the solution is smooth outside a closed singular set $S\subset \T^3$ with fractal dimension $\gamma <3$. Depending on how to model the fractal behaviour of such singularities, many physical models have been proposed to modify the K-41 theory (see for instance \cite[Chapter 2]{PV87}), for which the mathematical community subsequently tried to build solid analytical foundations (see \cite{Ey95,CS14} and references therein).  

Nevertheless, a precise statement on the energy conservation for such generalised Turbulence models, or better, a fractal counterpart of Onsager's conjecture,  seems, to the best of our knowledge, to be missing in the mathematical literature and it is indeed the content of this note. Let us point out that the question of energy conservation for weak solutions of \eqref{E} with a not space filling singular set has been addressed previously in \cite{S09,CKS97} but only for integer dimensions. The result of this note instead is more general and seems to be optimal since it matches with the predictions of some physical models and other mathematical results discussed below.


In 1978, Frisch, Sulem and Nelkin proposed in \cite{FSN78} a modification of the K-41 theory, called the $\beta-$model (also known as absolute curdling), in which the energy dissipation is assumed to be uniformly distributed on a homogeneous fractal with a given dimension $\gamma <3$. Under this assumption they deduce that the structure functions obey
\begin{equation}\label{Sfunc_beta}
S_p(\ell)\simeq \ell^{\zeta_p^*} \quad \text{with} \quad \zeta_p^*=3-\gamma+\frac{p}{3}(\gamma-2).
\end{equation}
Indeed, they deduce the validity of the previous relation as a consequence of \cite[Formula (3.5)]{FSN78} which asserts that a typical velocity increment over a distance $\ell$ is
\begin{equation}\label{holder_betamodel}
|v(x+\ell \vec{e})-v(x)|\simeq \ell^{\frac{\gamma-2}{3}}.
\end{equation}
We refer to \cite{Ey95,CS14} where the validity of such scaling laws has been mathematically consolidated for various fractal, or more generally \emph{multifractal}, models of Turbulence. Thus, in analogy to the heuristics leading from Kolmogorov's Theory (corresponding  to the case $\gamma=3$) to the Onsager exponent $\sfrac{1}{3}$, 
this fractal model suggests that the natural H\"older regularity exponent of a turbulent solution is given by $\theta=\sfrac{(\gamma-2)}{3}$, or, from the opposite perspective, given an H\"older regularity $\theta$ of the solution, the right fractal dimension which is consistent with the scaling law \eqref{holder_betamodel} is 
\begin{equation}
\label{Gamma_sharp}
\gamma=2+3\theta.
\end{equation}
Guided from these heuristics we prove 
a generalised energy conservation result for incompressible Euler equations, which can be viewed as the fractal generalisation of the proof of Constantin, E and Titi \cite{CET94}. 

In order to state the result, we introduce, for a time-dependent family of sets $\{S_{t} \}_{t \in [0, T]}\subseteq \T^3$, a uniform-in-time upper Minkowski dimension
\begin{equation}\label{def:Minkowskiunif2}
\overline{\dim}_{\mathcal{M}_{\infty}}\left( \{ S_{t} \}_{t \in [0, T]} \right):= \inf \left\{ s \geq 0: \, \overline{\mathcal{M}}^{s}_{\infty} \left( \{ S_{t} \}_{t \in [0, T]} \right)=0  \right\} \,,
\end{equation}
where $\overline{\mathcal{M}}^{ s}_{\infty}$ denotes a variant of the upper Minkowski content, obtained by computing the uniform-in-time shrinking rate (in volume) of the $\varepsilon$-neighbourhoods $\left(S_{t}\right)_{\varepsilon}:=\{ x \in \R^3: \, \dist(x,S_{t})\leq \varepsilon \}.$ More explicitly,  
\begin{equation}\label{def:Minkowskiunif}
\overline{\mathcal{M}}^{ s}_{\infty}\left( \{ S_{t} \}_{t \in [0, T]} \right):= \limsup_{\varepsilon \downarrow 0} \, \esssup_{t \in [0, T]} (2\varepsilon)^{s-3} \mathcal{H}^3\left( \left(S_{t}\right)_{\varepsilon} \right).
\end{equation}
The definitions \eqref{def:Minkowskiunif2}--\eqref{def:Minkowskiunif} are natural in that they coincide with the usual Minkowski  content and dimension (controlling in turn the box-counting and Hausdorff dimension, see for instance \cite[Chapter 5]{Mattila}) in case the set $S_{t}=S$ does not depend on time. We refer to Section \ref{sec:proof} for details. In this language we have the following 

\begin{theorem}\label{thm:main} 
Let $\theta \in \left(0, \sfrac{1}{3}\right)$ and $r \in \left(\sfrac{3}{(2 + 3 \theta)}, \infty\right]\,.$ Assume that $v \in L^3\left((0, T);B^\theta_{3r, \infty}(\T^3)\right) \cap C^0\left([0, T];L^2(\T^3)\right)$ is a weak solution of \eqref{E} such that for almost every $t \in [0, T]$ there exists a  closed set $S_{t} \subseteq \T^3$ with the following properties: 
\begin{enumerate}[label=(\roman*)]
\item \label{hyp1} the family $\{S_{t}\}_{t \in [0,T]}$ is not space filling in the sense that
\begin{equation}\label{e:dim}
\gamma:=\overline{\dim}_{\mathcal{M}_{\infty}}\left( \{ S_{t} \}_{t \in [0, T]} \right)< 3 \,.
\end{equation}
\item \label{hyp2} $v(t) \in C^1 \left(\T^3 \setminus S_{t}\right)$ with 
\begin{equation}\label{e:blowup}
\lvert \nabla v (x,t) \rvert \leq C(t) \dist(x, S_{t})^{-\kappa} \qquad \forall x \in \T^3 \setminus S_{t}
\end{equation} 
for some time dependent constant $C\in L^3((0,T)) $ and some
\begin{equation}\label{e:kappa}
0 \leq  \kappa < \frac{r-1}{r} \cdot \frac{3(1-\theta)}{1-3\theta}\,.
\end{equation}
\end{enumerate}
Then if 
\begin{equation}\label{e:gamma}
\gamma < \begin{cases}
 3 - \frac{r}{r-1} (1-3 \theta) \quad &\text{ if } \kappa \leq 1-\theta \,,\\
  3 - \frac{\kappa}{1-\theta} \cdot \frac{r}{r-1} (1-3 \theta) \quad &\text{ if } \kappa > 1-\theta \,,
 \end{cases}
\end{equation}
 $v$ conserves the kinetic energy, that is $ e_{v}(t)=e_v(0) $ for every $t\in [0,T]$.
\end{theorem}

The theorem is stated in the larger class of Besov regularity;  for $r=\infty$ the latter corresponds to the usual H\"older space $C^\theta$ and in this case, the theorem gives the dimension \eqref{Gamma_sharp} as the sharp threshold for energy conservation as predicted from the $\beta-$model under a restriction on the blow-up rate of $\nabla v(x)$ as $x$ approaches the singular set. Such an assumption was not included in the original $\beta-$model \cite{FSN78}, where only a qualitative behaviour of the solution outside the singular set was enough to describe the model. It is somehow curious that to deduce the threshold \eqref{Gamma_sharp} the assumption on the gradient reads
\begin{equation}\label{blowup_1-theta}
|\nabla v(x, t)|\lesssim \dist(x,S_t)^{-(1-\theta)} \qquad \forall x \in \T^3 \setminus S_{t} \,,
\end{equation}
which is on one side a common behaviour one would expect from reasonable $\theta-$H\"older continuous function and on the other side it is clearly correlated with the definition of the singular set  of order $\theta$, $S(\theta)$, given in the seminal paper \cite{FP85} of Frisch and Parisi when introducing the so-called \emph{multifractal} model.  We also refer to \cite{Ey95} in which upper bounds on the Hausdorff dimension of $S(\theta)$ are given in terms of the Besov regularity of the solution. 

The energy conservation result of Theorem \ref{thm:main}  is a consequence of the general fact that, under the assumptions of Theorem \ref{thm:main}, every time dependent vector field $v=v(x,t)$ (not necessarily solving \eqref{E}) satisfies $v\in L^3((0,T); B^{\sfrac{1}{3}+}_{3,\infty}(\T^3))$ (see Proposition \ref{p:main} below). This improved Besov regularity of $v$  then implies energy conservation by  \cite{CET94}. We refer to Section \ref{sec:proof} for a detailed discussion on the hypothesis assumed in Theorem \ref{thm:main} together with an heuristic Euler-based proof (à la Constantin, E and Titi) which highlights the importance of those assumptions. The fact that in our case energy conservation is implied by the property on general functions of being $B^{\sfrac{1}{3}+}_{3,\infty}(\T^3)$ is also consistent with \cite{CCFS08} in which the authors constructed a divergence-free vector field in $B^{\sfrac{1}{3}}_{3,\infty}(\T^3)$ whose energy flux (that is basically the limit as $\delta\rightarrow 0$ of the right hand side of \eqref{en_moll} below) is non-zero. Thus it would be very unlikely that energy would be conserved while not having $v(t)\in B^{\sfrac{1}{3}}_{3,\infty}(\T^3)$. Also, we mention that the blow-up rate \eqref{blowup_1-theta} is also implicitly assumed in \cite{CKS97} since it appears in their definition of the space $\text{Lip}(\alpha_0,\alpha_1,0,k)$, thus in this direction our theorem generalises \cite{CKS97} since we can also deal with bigger blow-up rates, as well as non-integer dimensions.

Clearly, when the set $S_t$ is empty, assumption $(ii)$ has to be intended by posing  $\dist(x,S_t)=1$,  which would read as 
\begin{equation}
\label{L3_Lip}
|\nabla v(x,t)|\leq C(t) \text{ with } C\in L^3((0,T)) \,.
\end{equation}
The latter implies $v\in L^3((0,T);\text{Lip}(\T^3))$ and energy conservation follows immediately from \cite{CET94}.

Theorem \ref{thm:main} also generalises (in space) the result \cite{Isett17} by Isett which asserts that space filling singularities ($\gamma=3$) are needed in order to allow solutions of \eqref{E} belonging the endpoint regularity class $L^3((0,T);B^{\sfrac{1}{3}}_{3r,\infty})$ with $r>1$ to not conserve the kinetic energy. This is indeed another way to look at Theorem \ref{thm:main} that we highlight in the following 

\begin{corollary}\label{coroll}
Let $\theta\in \left(0,\sfrac13\right)$. Assume that $v\in L^3\left((0,T);C^\theta(\T^3)\right)\cap C^0\left([0,T];L^2(\T^3)\right)$ is a non-conservative weak solution of \eqref{E} such that for almost every $t\in[0,T]$ there exists a closed set $S_t\subseteq\T^3$ such that 
\begin{equation}\label{hyp_grad_cor}
v(t)\in C^1(\T^3\setminus S_t)\qquad \text{with} \quad |\nabla v(x,t)|\leq C(t) \dist(x,S_t)^{-\kappa}\, \quad \forall x\in \T^3\setminus S_t
\end{equation}
for some time dependent constant $C\in L^3((0,T))$ and some $\kappa\leq 1-\theta$. Then 
$$
\overline{\dim}_{\mathcal{M}_\infty}\left(\{S_t \}_{t\in [0,T]}\right) \geq 2+3\theta\,.
$$
\end{corollary}
Even if the previous result could be stated in the larger class of the Besov regular solutions considered in Theorem \ref{thm:main}, we prefer to state it only in the H\"older case since it naturally matches with the (apparently) sharp dimension \eqref{Gamma_sharp}. 

It is then natural to ask if the  convex integration techniques can be adapted in order to prove the sharpness of the previous corollary, namely to construct dissipative H\"older continuous weak solutions of the incompressible Euler equations whose space singularities concentrates on a spatial set of small dimension. Such a concentration of singularities has been recently achieved in \cite{DH21} in the temporal domain, where the two present authors proved the existence of a non-trivial lower bound for the Hausdorff dimension of the singular set of times for $C^\theta(\T^3\times [0,T])$ weak solutions together with the existence of solutions whose (non-empty) singular set of times has quantifiably small (in terms of $\theta$) Hausdorff dimension.  This answered to a question posed in \cite{CL20} (see also \cite{BCV19} where the idea of concentrating singularities on a small set of times via a convex integration scheme as been first introduced).  We conclude by mentioning that the non-conservative weak solutions constructed in \cite{DH21} (as well as the ones in \cite{BCV19,CL20}) satisfy $S_t=\emptyset$ for almost every $t>0$, thus they do not fall in the assumptions of our Theorem \ref{thm:main} since the validity of \eqref{L3_Lip}  would imply energy conservation as already discussed.

\subsection*{Acknowledgements}
The authors acknowledge the support of the SNF Grant $182565$.

\section{Preliminaries}\label{sec:prelim}
Along the paper, we will consider $\T^3$ as spatial domain, identifying it with the 3-dimensional cube $[0,1]^3 \subset \R^3 $. Thus, for any vector field $f:\T^3\to\R^3$ or scalar $f: \T^3 \to \R$ we will always work with its periodic extension to the whole space.

\subsection*{Besov spaces} We define for $ \theta \in (0,1)$ and $r \in [1, \infty]$ the Besov space 
\begin{equation*}
B^\theta_{r, \infty}(\T^3):= \{ f \in L^r(\T^3): [f]_{B^\theta_{r, \infty}(\T^3)} < \infty \}\,,
\end{equation*}
where
\begin{equation*}
[f]_{B^\theta_{r, \infty}(\T^3)}:= \sup_{h \neq 0} \frac{\norm{f(h+ \cdot) - f(\cdot)}_{L^r(\T^3)}}{ \lvert h \rvert^\theta}\,.
\end{equation*}
$B^\theta_{r, \infty}(\T^3)$ is a Banach space when equipped with the norm $\norm{f}_{B^\theta_{3r, \infty}(\T^3)} := \norm{f}_{L^r(\T^3)}  + [f]_{B^\theta_{r, \infty}(\T^3)}\,.$ It is worth to note that in our context, there is no advantage in considering the more general Besov spaces $B^\theta_{r, s}(\T^3)$ with $s \in [1, \infty)$ (defined analogously by taking the $L^s$-norm in $h$) because of the embedding $B^\theta_{r, s}(\T^3) \hookrightarrow B^\theta_{r, \infty}(\T^3)\,.$ For $r= \infty$ on the other hand, $B^\theta_{\infty, \infty}(\T^3)$ coincides with the space of H\"older continuous functions $C^\theta(\T^3)\,.$ Whenever we consider a time-dependent vector field $v: \T^3 \times (0, T) \rightarrow \R^3$, we denote by $[v(t)]_{X}$ and $\norm{v(t)}_{X}$ spatial norms computed for fixed time $t \in (0, T)$ and by $L^p((0, T);X)$ the associated Bochner spaces for $p \in [1, \infty]$.

\subsection*{Mollification} Let $\varphi\in\mathcal C^{\infty}_c(\R^d)$ be a smooth, nonnegative and compactly supported function with $\|\varphi\|_{L^1}=1.$  For any $\delta>0$ we define $\varphi_\delta(x)=\delta^{-d}\varphi(x/\delta)$ and we consider, for a vector field $f:\T^d\to\R^d$, its spatial regularisation
$
f_\delta(x)=(f\ast\varphi_\delta)(x)=\int_{\R^d}f(x-y)\varphi_\delta(y)\,dy.
$ 
We recall some classical mollification estimates as well as the commutator estimate of \cite{CET94} which is crucial to prove the rigidity part of Onsager's conjecture.
\begin{proposition}\label{prop:molli} Let $d \geq 1$. For a vector field $f: \T^d\to \R^d\,,$ $\theta\in(0,1)$ and $r\in[1,\infty]\,,$ we have the following
\begin{align}\label{molli:2}
\|\nabla f_\delta\|_{L^r(\T^d)}&\lesssim \delta^{\theta-1}\|f\|_{B^\theta_{r,\infty}(\T^d)},\\\label{molli:3}
\|f_\delta\otimes f_\delta-(f\otimes f)_\delta\|_{L^r(\T^d)}&\lesssim \delta^{2\theta}\|f\|^2_{B^{\theta}_{2r,\infty}(\T^d)},
\end{align}
where the constants in the inequalities above may depend on $\theta$ and $r$ but are independent of $\delta$.
\end{proposition}

\subsection*{Constants} Throughout the paper we use the symbol $\lesssim$ to indicate that an inequality holds with a constant which may depend on all parameters of assumptions $\ref{hyp1}$ and $\ref{hyp2}$, i.e.  of the parameters $\gamma$, $\kappa$, $r$, $\theta$.

\section{Some remarks and Proof of Theorem \ref{thm:main}}\label{sec:proof}

Before giving the proof of our main result we list some remarks on the hypothesis of \ref{thm:main} as well as an heuristic computation which allows to guess, at least in the H\"older case, the threshold $\gamma=2+3\theta$.

\subsection*{Hypothesis \ref{hyp1}} In case that $S_{t}=S$ does not depend on time, the assumption \eqref{e:dim} reduces to the requirement that the usual upper Minkowski dimension $\gamma= \overline{\dim}_{\mathcal{M}}( S )<3\,,$ where we recall for completeness that
\begin{equation*}
\overline{\dim}_{\mathcal{M}} (S):= \inf \{ s \geq 0: \, \overline{\mathcal{M}}^{ s}(S)=0\} \qquad \text{ with } \qquad \overline{\mathcal{M}}^{ s} (S):= \limsup_{\varepsilon \downarrow 0} (2 \varepsilon)^{s-3} \mathcal{H}^3\left( ( S)_{\varepsilon} \right) \,.
\end{equation*}
In the time-dependent case, the notion of dimension \eqref{def:Minkowskiunif2} is more restrictive than taking the (essential) supremum in time of the upper Minkowski dimensions since the shrinking rate \eqref{def:Minkowskiunif} is required to hold uniform in time. Indeed, if hypothesis \eqref{e:dim} holds, then 
$$\overline{\mathcal{M}}^{ \bar \gamma}_{\infty} \left( \{ S_{t} \}_{t \in [0, T]} \right) =0 \qquad \text{ for every } \bar \gamma > \gamma$$
and hence in particular, there exists $\varepsilon_{0}= \varepsilon_{0}(\bar \gamma)\in (0,1)$ such that for a.e. $t \in [0,T]$
\begin{equation}\label{e:cond}
\mathcal{H}^3 \left( ( S_{t})_{\varepsilon}\right) \leq \varepsilon^{3- \bar \gamma} \qquad \forall\, \varepsilon \leq \varepsilon_{0} \,.
\end{equation}
It is now immediate to deduce that hypothesis  \eqref{e:dim} implies $\overline{\dim}_{\mathcal M}(S_{t}) \leq \gamma$ for almost every $t \in [0,T]$. In other words, it holds in general that 
$$
\overline{\dim}_{\mathcal{M}_\infty}\left(\{S_t\}_{t\in [0,T]}\right)\geq \esssup_{t\in [0,T]}\left( \overline{\dim}_{\mathcal M}(S_{t})\right).
$$
 Conversely, if we only require that  $\esssup_{t \in[0, T]}\dim_{\mathcal{M}} (S_{t}) \leq \gamma <3 $, then the threshold $\varepsilon_{0}= \varepsilon_{0}(\bar \gamma, t)$ in \eqref{e:cond} can in general not assumed to be uniform in time and this is in fact the main reason for introducing the uniform-in-time Minkowski dimension \eqref{def:Minkowskiunif}.

\subsection*{Restrictions on $r$ and $\kappa$} The lower bound $r > \sfrac{3}{(2+3 \theta)}>1$, as well as the restriction \eqref{e:kappa} on $\kappa$, is a simple compatibility criterion which guarantees that the right hand-side of \eqref{e:gamma} is strictly positive; otherwise the statement of Theorem \ref{thm:main} is empty. Observe that $r \downarrow 1$ as $\theta \uparrow \sfrac{1}{3}$, which is compatible with the Besov regularity classes considered by \cite{Isett17} at the sharp Onsager exponent $\theta=\sfrac{1}{3}\,.$

\subsection*{Dimension $d \geq 2$} Theorem \ref{thm:main} and its proof generalise easily to any dimension $d \geq 2$ after the following straightforward changes in the statement: in assumption $(i)$ we require the singular set to be not space filling, meaning that $\gamma<d$. In assumption \ref{hyp2} the restriction on $\kappa$ should now read $0\leq \kappa <  \frac{r-1}{r} \cdot \frac{d(1-\theta)}{1-3\theta} \,.$  Correspondingly, the upper bound on $\gamma$ in \eqref{e:gamma} should be replaced by 
$$
\gamma < \begin{cases}
 d - \frac{r}{r-1} (1-3 \theta) \quad &\text{ if } \kappa \leq 1-\theta \,, \\
  d - \frac{\kappa}{1-\theta} \cdot \frac{r}{r-1} (1-3 \theta) \quad &\text{ if } \kappa > 1-\theta \,.
 \end{cases}
\,$$

\subsection*{Heuristic proof à la Constantin, E and Titi \cite{CET94}} To motivate the validity of the threshold $\gamma=2+3\theta$ for the upper Minkowski dimension of the spatial singular set, we give a heuristic proof of energy conservation à la Constantin, E and Titi in the easier case in which $v\in L^3((0,T);C^\theta(\T^3))$ and the singular set $S\subseteq \T^3$  is not depending on time. This proof can be made rigorous and was actually our original idea to tackle the problem; only in the end we realised that the proof of Theorem \ref{thm:main} can be reduced to a general property of Besov functions (see Proposition \ref{p:main} below).

The well-known energy conservation proof \cite{CET94} is based on spatial regularisations $(v_{\delta}, p_{\delta})$ of $(v, p)$ at scale $\delta>0$, which solve the Euler-Reynolds system 
\begin{equation*}
\partial_{t} v_{\delta} + \div( v_{\delta} \otimes v_{\delta}) + \nabla p_{\delta} = \div R_{\delta}
\end{equation*}
with $R_{\delta}:= v_{\delta} \otimes v_{\delta} - (u \otimes u)_{\delta}$. Scalar multiplying the previous regularised equation by $v_\delta$ and integrating on $\T^3$ we obtain
\begin{equation}\label{en_moll}
\frac{d}{dt} e_{v_{\delta}}(t)= - \int_{\T^3} R_{\delta} \cdot \nabla v_{\delta} \, dx.
\end{equation}
Since we assumed the solution $v$ to be $C^1$ outside $S$, it is convenient to split the scalar product in the right-hand side of \eqref{en_moll} in $(S)_\varepsilon :=\{ x \in \R^3: \, \dist(x,S)\leq \varepsilon \}$ and in its complement
$$
\frac{d}{dt} e_{v_{\delta}}(t)= - \int_{(S)_\varepsilon} R_{\delta} \cdot \nabla v_{\delta} \, dx-\int_{\T^3 \setminus S_\varepsilon} R_{\delta} \cdot \nabla v_{\delta} \, dx.
$$
In order to look only at regular points of $R_\delta:\nabla v_\delta$ in $\T^3 \setminus S_\varepsilon$, we need to keep the mollification distant from the singular set $S$, which means that 
\begin{equation}\label{eps_delt_compatibility}
\delta\lesssim \varepsilon.
\end{equation}
From the assumption $\dim_\mathcal{M}(S)=\gamma<3$ we deduce $\mathcal{H}^3((S)_\varepsilon)\lesssim \varepsilon^{3-\gamma}$  for all sufficiently small $\varepsilon>0$, which together with \eqref{molli:2} and \eqref{molli:3} yields 
\begin{equation}\label{gamma_threshold_heur}
\left|\int_{(S)_\varepsilon} R_{\delta} \cdot \nabla v_{\delta} \, dx\right|\lesssim [v(t)]^3_{C^\theta(\T^3)} \delta^{3\theta-1}\varepsilon^{3-\gamma}.
\end{equation}
Since $\theta<\sfrac13$, the exponent of $\delta$ in the right-hand side of \eqref{gamma_threshold_heur} is negative and hence the best choice of $\delta$ (in terms of $\varepsilon$) is the biggest admissible one. The compatibility condition \eqref{eps_delt_compatibility} implies $\delta\simeq \varepsilon$, which gives
$$
\left|\int_{(S)_\varepsilon} R_{\varepsilon} \cdot \nabla v_{\varepsilon} \, dx\right|\lesssim [v(t)]^3_{C^\theta(\T^3)} \varepsilon^{2+3\theta-\gamma}.
$$
As $\varepsilon \rightarrow 0$, this term goes to zero if $\gamma<2+3\theta$,  which after integration in time, implies energy conservation provided that also the second term converges to zero.  The latter can be estimated by 
$$
\left|\int_{\T^3 \setminus (S)_\varepsilon} R_{\varepsilon} \cdot \nabla v_{\varepsilon} \, dx\right|\lesssim [v(t)]^2_{C^\theta(\T^3)}\varepsilon^{2\theta} \int_{\T^3 \setminus S_\varepsilon}|\nabla v_\epsilon|\,dx.
$$
For a general H\"older continuous function the quantity  $  \int_{\T^3 \setminus S_\varepsilon}|\nabla v_\epsilon|\,dx$ is not bounded in $\varepsilon$ and more precisely, its blow-up rate will depend on the asymptotic behaviour of $|\nabla v(x,t)|$ as $x$ approaches the singular set. That is why, in order to conclude energy conservation, we also need such a hypothesis which is made precise in $\ref{hyp2}$.

\subsection*{Proof of Theorem \ref{thm:main}} Once an assumption on the blow-up rate of the gradient of the solution when approaching the singular set is needed, the proof of Theorem \ref{thm:main} can be reduced to a general property of Besov functions which is stated in the next proposition. The energy conservation of Theorem \ref{thm:main} then follows from the result of Constantin, E and Titi which asserts that weak solutions of Euler in the class $L^3((0,T);B^{\sfrac{1}{3} +}_{3,\infty}(\T^3))\cap C^0([0,T];L^2(\T^3))$ conserve the kinetic energy, where 
$$
B^{\frac{1}{3} +}_{3,\infty}(\T^3)=\bigcup_{\alpha>0}B^{\frac{1}{3} +\alpha}_{3,\infty}(\T^3) \,.
$$

\begin{proposition}\label{p:main}
Let $\theta, r, v, S_t, \gamma, \kappa$ as in the statement of Theorem \ref{thm:main} with the only exception that the time dependent vector field $v=v(x,t)$ does not necessarily solve Euler. 
Then if 
\begin{equation}\label{e:gamma_p}
\gamma < \begin{cases}
 3 - \frac{r}{r-1} (1-3 \theta) \quad &\text{ if } \kappa \leq 1-\theta \,, \\
  3 - \frac{\kappa}{1-\theta} \cdot \frac{r}{r-1} (1-3 \theta) \quad &\text{ if } \kappa > 1-\theta \,,
 \end{cases}
\end{equation}
we have $v\in L^3((0,T);B^{\sfrac{1}{3}+}_{3,\infty}(\T^3))$.
\end{proposition}
\begin{proof}
We fix $\bar \gamma\in (\gamma, 3)$ (which will be chosen sufficiently close to $\gamma $ at the very end) and the corresponding threshold $\varepsilon_0=\varepsilon_0(\bar \gamma)$ given by hypothesis $\ref{hyp1}$ such that \eqref{e:cond} holds for almost every $t\in (0,T)$. 
Being $\varepsilon_0$ uniform in time, it is enough to prove that $\exists \,\alpha>0$ such that 
\begin{equation}\label{cond_B_+}
\int_0^T \|v(\cdot +h,t)-v(\cdot, t)\|^3_{L^3(\T^3)}\,dt \lesssim |h|^{1+\alpha}\quad \forall \,|h|<\frac{\varepsilon_0}{2}.
\end{equation}
Thus let us fix any $h\in B_{\sfrac{\varepsilon_0}{2}}(0)$. For every $\varepsilon>0$ such that 
\begin{equation}
\label{cond_h_eps}
2|h|\leq \varepsilon<\varepsilon_0
\end{equation}
we consider the neighbourhoods $(S_t)_\varepsilon=\{ x\in \T^3\, :\, \dist(x,S_t)\leq \varepsilon\}$ and we split
\begin{align*}
\|v(\cdot +h,t)-v(\cdot, t)\|^3_{L^3(\T^3)}&= \int_{(S_t)_\varepsilon} |v(x +h,t)-v(x, t)|^3\,dx+ \int_{(S_t)_{\varepsilon_0}\setminus (S_t)_{\varepsilon}} |v(x +h,t)-v(x, t)|^3\,dx\\
&+ \int_{\T^3\setminus (S_t)_{\varepsilon_0}} |v(x +h,t)-v(x, t)|^3\,dx=I+II+III.
\end{align*}
Note that the compatibility condition \eqref{cond_h_eps} ensures that in the integrals $II$ and $III$ we only see regular points of $v$. 

We estimate $I$ for a.e. $t \in (0,T)$ by using \eqref{e:cond} together with the assumption $v(t)\in B^\theta_{3r,\infty}(\T^3)$
\begin{equation}
\label{est_I}
I\leq \left( \int_{(S_t)_\varepsilon}  |v(x +h,t)-v(x, t)|^{3r}\,dx \right)^\frac{1}{r}\mathcal{H}^3\left((S_t)_\varepsilon \right)^{\frac{r-1}{r}}\leq [v(t)]_{B^{\theta}_{3r,\infty}(\T^3)}^3 |h|^{3\theta}\varepsilon^{(3-\bar \gamma)\frac{r-1}{r}}.
\end{equation}
From assumption $\ref{hyp2}$, together with the constraint \eqref{cond_h_eps}, we deduce that $\forall \, s\in[0,1]$ and every $x\in \T^3\setminus (S_t)_\varepsilon$ we have the bound
\begin{equation}
\label{est_grad}
|\nabla v(x+sh,t)|\leq C(t) \dist (x+sh,S_t)^{-\kappa}\leq C(t)\left( \dist(x,S_t)-|h|\right)^{-\kappa}\leq C(t)2^\kappa \dist (x,S_t)^{-\kappa}.
\end{equation}
Combining \eqref{est_grad} with the $C^1$-regularity of $v$ on $\T^3 \setminus (S_{t})_{\varepsilon}$ allows to estimate
\begin{equation}
\label{est_III}
III\leq |h|^3 \int_0^1 \int_{\T^3\setminus (S_t)_{\varepsilon_0}}|\nabla v(x+sh,t)|^3 \,dx\,ds\lesssim C^3(t)|h|^3 \varepsilon_{0}^{-3\kappa} \,.
\end{equation}
We are left with the second term $II$ which is indeed the only non-trivial estimate. By using again \eqref{est_grad}, we have by H\"older
\begin{align}
II&\leq \left(\int_{\T^3}  |v(x +h,t)-v(x, t)|^{2r}\,dx\right)^{\frac{1}{r}} \left(\int_{(S_t)_{\varepsilon_0}\setminus (S_t)_{\varepsilon}}  |v(x +h,t)-v(x, t)|^{\frac{r}{r-1}}\,dx\right)^{\frac{r-1}{r}}\nonumber\\
&\lesssim |h|^{2\theta}[v(t)]^2_{B^\theta_{2r,\infty}(\T^3)}C(t)|h|\left(\int_{(S_t)_{\varepsilon_0}\setminus (S_t)_{\varepsilon}} \dist(x,S_t)^{-\frac{r}{r-1}\kappa}\,dx\right)^{\frac{r-1}{r}}.\label{est_II}
\end{align}
To estimate the last integral we dyadically decompose the set 
$$
(S_t)_{\varepsilon_0}\setminus (S_t)_{\varepsilon}\subset \bigcup_{i=j_{0}}^{j} \left \{ x \in \T^3: 2^{-i} \leq \dist(x, S_{t})\leq 2^{-(i-1)} \right\}  =: \bigcup_{i=j_{0}}^{j} A_{t, i} \,,
$$
for some fixed $j_0\in \mathbb{N}$ (depending only on $\varepsilon_0$) and $j\in \mathbb{N}$, $j\geq j_0$ which satisfies 
\begin{equation}
\label{eps_i}
\frac{\varepsilon}{2}\leq 2^{-j}\leq \varepsilon.
\end{equation}
Thus, by using again \eqref{e:cond},  we have 
\begin{align*}
\int_{(S_t)_{\varepsilon_0}\setminus (S_t)_{\varepsilon}} \dist(x,S_t)^{-\frac{r}{r-1}\kappa}\,dx&\leq \sum_{i=j_0}^{j} \int_{A_{t,i}} \dist(x,S_t)^{-\frac{r}{r-1}\kappa}\,dx\leq \sum_{i=j_0}^{j}2^{i\frac{r}{r-1}\kappa}\mathcal{H}^3(A_{t,i})\\
&\leq  \sum_{i=j_0}^{j} 2^{i\frac{r}{r-1}\kappa}\mathcal{H}^3\left((S_t)_{2^{-(i-1)}}\right)\lesssim  \sum_{i=0}^{j}2^{\left(\frac{r}{r-1}\kappa - (3-\bar \gamma)\right)i},
\end{align*}
from which we deduce that 
$$
\left(\int_{(S_t)_{\varepsilon_0}\setminus (S_t)_{\varepsilon}} \dist(x,S_t)^{-\frac{r}{r-1}\kappa}\,dx\right)^{\frac{r-1}{r}}\lesssim \sum_{i=0}^j 2^{\left(\kappa - (3-\bar \gamma)\frac{r-1}{r}\right)i}\,.
$$
The behaviour of the previous sum depends on the sign of the exponent $\kappa- (3-\bar \gamma)\frac{r-1}{r}$, but we can bound it by
$$
\sum_{i=0}^j 2^{\left(\kappa - (3-\bar \gamma)\frac{r-1}{r}\right)i}\leq C_{\kappa,\bar \gamma, r} \, j \, 2^{j\max (0,\kappa- (3-\bar \gamma)\frac{r-1}{r})}\lesssim | \log \varepsilon|\varepsilon^{-\max \left(0,\kappa- (3-\bar \gamma)\frac{r-1}{r}\right)},
$$
where we used the choice \eqref{eps_i} in the last inequality.  Inserting this last estimate in \eqref{est_II}, we obtain 
$$
II\lesssim C(t)[v(t)]^2_{B^\theta_{3r,\infty}(\T^3)}|h|^{1+2\theta}|\log \varepsilon|\varepsilon^{-\max \left(0,\kappa- (3-\bar \gamma)\frac{r-1}{r}\right)},
$$
which together with \eqref{est_I} and \eqref{est_III}, by also integrating in time, yields for $2\lvert h \rvert < \varepsilon_{0}$
\begin{align}
\int_0^T \|v(\cdot +h,t)-v(\cdot, t)\|^3_{L^3(\T^3)}\,dt &\lesssim[v]^3_{L^3((0,T);B^\theta_{3r,\infty}(\T^3))}|h|^{3\theta}\varepsilon^{(3-\bar \gamma)\frac{r-1}{r}}+\|C\|^3_{L^3((0,T))}|h|^3\nonumber \\
&+\|C\|_{L^3((0,T))}[v]^2_{L^3((0,T);B^\theta_{3r,\infty}(\T^3))} |h|^{1+2\theta}\left| \log \varepsilon\right| \varepsilon^{-\max \left(0,\kappa- (3-\bar \gamma)\frac{r-1}{r}\right)}.\label{est_ugly}
\end{align}
Finally, if $\kappa\leq 1-\theta$, we conclude by choosing $\bar \gamma>\gamma$ such that 
$$
3-\frac{r}{r-1}(1-\theta)<\bar\gamma<3-\frac{r}{r-1}(1-3\theta)
$$
so that, with the choice $\varepsilon=2|h|$, \eqref{est_ugly} reduces to
\begin{align*}
\int_0^T \|v(\cdot +h,t)-v(\cdot, t)\|^3_{L^3(\T^3)}\,dt &\lesssim |h|^{3\theta+(3-\bar\gamma)\frac{r-1}{r}}+|h|^3+|h|^{1+2\theta}\left| \log |h|\right| |h|^{(3-\bar\gamma)\frac{r-1}{r}-(1-\theta)}\\
&\lesssim \left| \log |h|\right| |h|^{3\theta+(3-\bar\gamma)\frac{r-1}{r}}+|h|^3,
\end{align*}
which implies \eqref{cond_B_+} by choosing $\bar \gamma$ sufficiently close to $\gamma$, if $\gamma<3-\frac{r}{r-1}(1-3\theta)$.

If instead, $\kappa>1-\theta$, we can enforce 
$$
3-\frac{r}{r-1}\kappa<\bar\gamma<3-\frac{r}{r-1}\frac{\kappa}{1-\theta}(1-3\theta)
$$
so that, with the choice\footnote{Note that this choice is compatible with the restriction $2|h|<\varepsilon$ since $\frac{1-\theta}{\kappa}<1$. Also, note that this choice of $\varepsilon$ could exceed $\varepsilon_0$, but this can be avoided if $|h|<\frac{\varepsilon_0}{2}$ was chosen sufficiently small depending on $\kappa$ and $\theta$. } $\varepsilon=2|h|^\frac{1-\theta}{\kappa}$, we reduces \eqref{est_ugly} to
\begin{align*}
\int_0^T \|v(\cdot +h,t)-v(\cdot, t)\|^3_{L^3(\T^3)}\,dt &\lesssim |h|^{3\theta+(3-\bar\gamma)\frac{r-1}{r}\frac{1-\theta}{\kappa}}+|h|^3+|h|^{1+2\theta}\left| \log |h|\right| |h|^{\frac{1-\theta}{\kappa}\left((3-\bar\gamma)\frac{r-1}{r}-\kappa)\right)}\\
&\lesssim \left| \log |h|\right| |h|^{3\theta+(3-\bar\gamma)\frac{1-\theta}{\kappa}\frac{r-1}{r}}+|h|^3,
\end{align*}
which again implies \eqref{cond_B_+} by choosing $\bar \gamma$ sufficiently close to $\gamma$, if $\gamma<3-\frac{r}{r-1}\frac{\kappa}{1-\theta}(1-3\theta)$.

\end{proof}



\begin{bibdiv}
\begin{biblist}


\bib{BJPV98}{book}{
   author={Bohr, Tomas},
   author={Jensen, Mogens H.},
   author={Paladin, Giovanni},
   author={Vulpiani, Angelo},
   title={Dynamical systems approach to turbulence},
   series={Cambridge Nonlinear Science Series},
   volume={8},
   publisher={Cambridge University Press, Cambridge},
   date={1998},
   pages={xx+350},
   isbn={0-521-47514-7},
   review={\MR{1643112}},
   doi={10.1017/CBO9780511599972},
}



\bib{BDLIS15}{article}{
   author={Buckmaster, Tristan},
   author={De Lellis, Camillo},
   author={Isett, Philip},
   author={Sz\'{e}kelyhidi, L\'{a}szl\'{o}, Jr.},
   title={Anomalous dissipation for $1/5$-H\"{o}lder Euler flows},
   journal={Ann. of Math. (2)},
   volume={182},
   date={2015},
   number={1},
   pages={127--172},
   issn={0003-486X},
   review={\MR{3374958}},
   doi={10.4007/annals.2015.182.1.3},
}

\bib{BDS16}{article}{
   author={Buckmaster, Tristan},
   author={De Lellis, Camillo},
   author={Sz\'{e}kelyhidi, L\'{a}szl\'{o}, Jr.},
   title={Dissipative Euler flows with Onsager-critical spatial regularity},
   journal={Comm. Pure Appl. Math.},
   volume={69},
   date={2016},
   number={9},
   pages={1613--1670},
   issn={0010-3640},
   review={\MR{3530360}},
   doi={10.1002/cpa.21586},
}

\bib{BDLSV2019}{article}{
   author={Buckmaster, Tristan},
   author={De Lellis, Camillo},
   author={Sz\'{e}kelyhidi, L\'{a}szl\'{o}, Jr.},
   author={Vicol, Vlad},
   title={Onsager's conjecture for admissible weak solutions},
   journal={Comm. Pure Appl. Math.},
   volume={72},
   date={2019},
   number={2},
   pages={229--274},
   issn={0010-3640},
   review={\MR{3896021}},
   doi={10.1002/cpa.21781},
}
	
\bib{BCV19}{article}{
   author={Buckmaster, Tristan},
   author={Colombo, Maria},
   author={Vicol, Vlad},
   title={Wild solutions of the Navier-Stokes equations whose singular sets in time have Hausdorff dimension strictly less than 1},
   journal={	arXiv:1809.00600 [math.AP]},
   date={2019},
}
	

\bib{CKS97}{article}{
   author={Caflisch, Russel E.},
   author={Klapper, Isaac},
   author={Steele, Gregory},
   title={Remarks on singularities, dimension and energy dissipation for
   ideal hydrodynamics and MHD},
   journal={Comm. Math. Phys.},
   volume={184},
   date={1997},
   number={2},
   pages={443--455},
   issn={0010-3616},
   review={\MR{1462753}},
   doi={10.1007/s002200050067},
}


\bib{CCFS08}{article}{
   author={Cheskidov, A.},
   author={Constantin, P.},
   author={Friedlander, S.},
   author={Shvydkoy, R.},
   title={Energy conservation and Onsager's conjecture for the Euler
   equations},
   journal={Nonlinearity},
   volume={21},
   date={2008},
   number={6},
   pages={1233--1252},
   issn={0951-7715},
   review={\MR{2422377}},
   doi={10.1088/0951-7715/21/6/005},
}

\bib{CL20}{article}{
   author={Cheskidov, Alexey},
   author={Luo, Xiaoyutao},
   title={Sharp nonuniqueness for the Navier-Stokes equations},
   journal={	arXiv:2009.06596 [math.AP]},
   date={2020},
}

\bib{CS14}{article}{
   author={Cheskidov, A.},
   author={Shvydkoy, R.},
   title={Euler equations and turbulence: analytical approach to
   intermittency},
   journal={SIAM J. Math. Anal.},
   volume={46},
   date={2014},
   number={1},
   pages={353--374},
   issn={0036-1410},
   review={\MR{3152734}},
   doi={10.1137/120876447},
}

\bib{CD18}{article}{
   author={Colombo, Maria},
   author={De Rosa, Luigi},
   title={Regularity in time of H\"{o}lder solutions of Euler and
   hypodissipative Navier-Stokes equations},
   journal={SIAM J. Math. Anal.},
   volume={52},
   date={2020},
   number={1},
   pages={221--238},
   issn={0036-1410},
   review={\MR{4051979}},
   doi={10.1137/19M1259900},
}



		
\bib{CET94}{article}{
   author={Constantin, Peter},
   author={E, Weinan},
   author={Titi, Edriss S.},
   title={Onsager's conjecture on the energy conservation for solutions of
   Euler's equation},
   journal={Comm. Math. Phys.},
   volume={165},
   date={1994},
   number={1},
   pages={207--209},
   issn={0010-3616},
   review={\MR{1298949}},
}

\bib{DS17}{article}{
    author = {Daneri, Sara},
    author ={ Sz\'{e}kelyhidi Jr., L\'{a}szl\'{o}},
    title = {Non-uniqueness and h-principle for {H}\"{o}lder-continuous weak
              solutions of the {E}uler equations},
   journal = {Arch. Ration. Mech. Anal.},
    volume = {224},
      date = {2017},
    number = {2},
     pages = {471--514},
      issn = {0003-9527},
  review = {\MR{3614753}},
       doi = {10.1007/s00205-017-1081-8},
}



\bib{DS2013}{article}{
   author={De Lellis, Camillo},
   author={Sz\'{e}kelyhidi, L\'{a}szl\'{o}, Jr.},
   title={Dissipative continuous Euler flows},
   journal={Invent. Math.},
   volume={193},
   date={2013},
   number={2},
   pages={377--407},
   issn={0020-9910},
   review={\MR{3090182}},
   doi={10.1007/s00222-012-0429-9},
}
		
\bib{DR19}{article}{
   author={De Rosa, Luigi},
   title={Infinitely many Leray-Hopf solutions for the fractional
   Navier-Stokes equations},
   journal={Comm. Partial Differential Equations},
   volume={44},
   date={2019},
   number={4},
   pages={335--365},
   issn={0360-5302},
   review={\MR{3941228}},
   doi={10.1080/03605302.2018.1547745},
}


\bib{DH21}{article}{
   author={De Rosa, Luigi},
    author={Haffter, Silja},
   title={Dimension of the singular set of wild H\"older solutions of the incompressible Euler equations},
   journal={	arXiv:2102.06085 [math.AP]},
   date={2021},
   doi={Preprint},
}

\bib{DT20}{article}{
   author={De Rosa, Luigi},
    author={Tione, Riccardo},
   title={Sharp energy regularity and typicality results for Hölder solutions of incompressible Euler equations},
   journal={Analysis and  PDE},
   date={2020},
   doi={Accepted},
}

\bib{Ey94}{article}{
   author={Eyink, Gregory L.},
   title={Energy dissipation without viscosity in ideal hydrodynamics. I.
   Fourier analysis and local energy transfer},
   journal={Phys. D},
   volume={78},
   date={1994},
   number={3-4},
   pages={222--240},
   issn={0167-2789},
   review={\MR{1302409}},
   doi={10.1016/0167-2789(94)90117-1},
}

\bib{Ey95}{article}{
   author={Eyink, Gregory L.},
   title={Besov spaces and the multifractal hypothesis},
   note={Papers dedicated to the memory of Lars Onsager},
   journal={J. Statist. Phys.},
   volume={78},
   date={1995},
   number={1-2},
   pages={353--375},
   issn={0022-4715},
   review={\MR{1317149}},
   doi={10.1007/BF02183353},
}

\bib{ET99}{article}{
   author={Eyink, Gregory L.},
   author={Thomson, David J.},
   title={Free decay of turbulence and breakdown of self-similarity},
   journal={Phys. Fluids},
   volume={12},
   date={2000},
   number={3},
   pages={477--479},
   issn={1070-6631},
   review={\MR{1743086}},
   doi={10.1063/1.870279},
}


\bib{F95}{article}{
   author={Frisch, U.},
   title={Turbulence: The Legacy of A. N. Kolmogorov},
   journal={Cambridge: Cambridge University Press},
   date={1995},
   doi={10.1017/CBO9781139170666},
}



\bib{FP85}{article}{
   author={Frisch, U.},
   author={Parisi, G.},
   title={On the singularity structure of fully developed turbulence},
   journal={Turbulence and Predictability of Geophysical Flows and Climate Dynamics, (North- Holland, Amsterdam)},
   date={1985},
  pages={84-87},
}

\bib{FSN78}{article}{
   author={Frisch, U.},
   author={Sulem, P.},
      author={Nelkin, M.},
   title={A simple dynamical model of intermittent fully developed turbulence.},
   journal={Journal of Fluid Mechanics},
   volume={87},
   date={1978},
   number={4},
   pages={719-736},
   doi={10.1017/S0022112078001846},
}


\bib{Is2013}{article}{
   author={Isett, Philip},
   title={Regularity in time along the coarse scale flow for the incompressible Euler equations},
   journal={arXiv preprint, https://arxiv.org/abs/1307.0565},
   date={2018},
}

\bib{Isett17}{article}{
   author={Isett, Philip},
   title={On the Endpoint Regularity in {O}nsager's Conjecture},
   journal={arXiv preprint, https://arxiv.org/abs/1706.01549},
   date={2017},
}



\bib{Is2018}{article}{
   author={Isett,Philip},
   title={A proof of Onsager's conjecture},
   journal={Ann. of Math. (2)},
   volume={188},
   date={2018},
   number={3},
   pages={871--963},
   issn={0003-486X},
   review={\MR{3866888}},
   doi={10.4007/annals.2018.188.3.4},
}



\bib{K41}{article}{
   author={Kolmogorov, Andrej Nikolaevič},
   title={The local structure of turbulence in an incompressible viscous fluid},
   journal={Acad. Sci. URSS (N.S.)},
   date={1941},
   number={30},
   pages={301--305},
}

\bib{Mattila}{book}{
    AUTHOR = {Mattila, Pertti},
     TITLE = {Geometry of sets and measures in {E}uclidean spaces},
    SERIES = {Cambridge Studies in Advanced Mathematics},
    VOLUME = {44},
 PUBLISHER = {Cambridge University Press, Cambridge},
      YEAR = {1995},
     PAGES = {xii+343},
      ISBN = {0-521-46576-1; 0-521-65595-1},
   MRCLASS = {28A75 (49Q20)},
  MRNUMBER = {1333890},
MRREVIEWER = {Harold Parks},
       DOI = {10.1017/CBO9780511623813},
       URL = {https://doi.org/10.1017/CBO9780511623813},
}



		



\bib{Ons}{article}{
   author={Onsager, L.},
   title={Statistical hydrodynamics},
   journal={Nuovo Cimento (9)},
   volume={6},
   date={1949},
   number={Supplemento, 2 (Convegno Internazionale di Meccanica
   Statistica)},
   pages={279--287},
   issn={0029-6341},
   review={\MR{36116}},
}

\bib{PV87}{article}{
   author={Paladin, Giovanni},
   author={Vulpiani, Angelo},
   title={Anomalous scaling laws in multifractal objects},
   journal={Phys. Rep.},
   volume={156},
   date={1987},
   number={4},
   pages={147--225},
   issn={0370-1573},
   review={\MR{919714}},
   doi={10.1016/0370-1573(87)90110-4},
}

\bib{Sch93}{article}{
   author={Scheffer, Vladimir},
   title={An inviscid flow with compact support in space-time},
   journal={J. Geom. Anal.},
   volume={3},
   date={1993},
   number={4},
   pages={343--401},
   issn={1050-6926},
   review={\MR{1231007}},
   doi={10.1007/BF02921318},
}

\bib{Sh00}{article}{
   author={Shnirelman, A.},
   title={Weak solutions with decreasing energy of incompressible Euler
   equations},
   journal={Comm. Math. Phys.},
   volume={210},
   date={2000},
   number={3},
   pages={541--603},
   issn={0010-3616},
   review={\MR{1777341}},
   doi={10.1007/s002200050791},
}


\bib{S09}{article}{
   author={Shvydkoy, Roman},
   title={On the energy of inviscid singular flows},
   journal={J. Math. Anal. Appl.},
   volume={349},
   date={2009},
   number={2},
   pages={583--595},
   issn={0022-247X},
   review={\MR{2456214}},
   doi={10.1016/j.jmaa.2008.09.007},
}

\bib{S81}{article}{
   author={Siggia, E.},
   title={Numerical study of small-scale intermittency in three-dimensional turbulence.},
   journal={Journal of Fluid Mechanics},
   volume={107},
   date={1982},
   pages={375-406},
   doi={10.1017/S002211208100181X},
}

\end{biblist}
\end{bibdiv}
\end{document}